\documentclass[11pt,a4paper]{article}

\usepackage[utf8]{inputenc}
\usepackage[T1]{fontenc}
\usepackage{amsmath,amssymb,amsthm}
\usepackage{mathtools}
\usepackage[margin=2.5cm]{geometry}
\usepackage[hidelinks]{hyperref}

\theoremstyle{plain}
\newtheorem{theorem}{Theorem}[section]
\newtheorem{proposition}[theorem]{Proposition}
\newtheorem{lemma}[theorem]{Lemma}
\newtheorem{corollary}[theorem]{Corollary}
\theoremstyle{definition}
\newtheorem{definition}[theorem]{Definition}
\theoremstyle{remark}
\newtheorem{remark}[theorem]{Remark}

\newcommand{\R}{\mathbb{R}}
\newcommand{\dx}{\,\mathrm{d}x}
\newcommand{\ds}{\,\mathrm{d}s}
\newcommand{\n}[1]{\left\| #1 \right\|}
\DeclareMathOperator{\Id}{Id}
\DeclareMathOperator{\Ran}{Ran}

\begin{document}

\title{\textbf{Two-sided estimates of the blow-up time for a semilinear
wave equation with fractional structural damping}}

\author{Firas Kaabi\\[4pt]
\normalsize Faculty of Sciences of Tunis, University of Tunis El Manar\\
\normalsize LR Analyse Non-Lin\'eaire et G\'eom\'etrie, LR21ES08,
El Manar 2, 2092, Tunis, Tunisia\\
\normalsize \texttt{firaskaabi17@gmail.com}}

\date{}

\maketitle

\begin{abstract}
We consider the initial--boundary value problem for the semilinear wave
equation with fractional structural damping
$$
u_{tt}+(-\Delta)^{\theta}u_{t}-\Delta u=|u|^{p-1}u
$$
in a bounded domain, where the exponent $\theta\in[0,1]$ interpolates
between external frictional damping ($\theta=0$) and internal Kelvin--Voigt
viscoelastic damping ($\theta=1$), and therefore parametrises the
frequency dependence of the dissipation mechanism. For initial data with
energy below the depth of the potential well and negative Nehari functional
we prove finite-time blow-up together with an explicit upper bound for the
blow-up time. The bound comes from a single concavity functional in which
the fractional dissipation cancels identically; it therefore has the same
form for every $\theta\in[0,1]$, and it admits a variant that is uniform in
$\theta$. Conversely, for $1<p\le\frac{n+2\theta}{n-2}$ when $n\ge3$, we
establish an explicit lower bound for the blow-up time. Mechanically, the
two bounds delimit a guaranteed interval of existence and a guaranteed
failure time for the model. The admissible range of exponents in the lower
bound widens linearly with $\theta$, which quantifies how the strength of
the internal damping enlarges the class of nonlinear loads for which such a
guarantee can be computed; we do not claim monotonicity in $\theta$ of the
numerical value of the bound. Both theorems are proved for a class of energy
solutions specified by a short list of requirements, so that they are
independent of any particular local existence theorem, and they carry over
unchanged to the elasticity and plate operators used in structural models.
The two endpoint cases recover known results for frictional and strong
damping.


\end{abstract}

\section{Introduction}\label{sec:intro}

Let $\Omega\subset\R^{n}$, $n\ge1$, be a bounded open set and let
$\theta\in[0,1]$ be fixed. We study the initial--boundary value problem
\begin{equation}
\begin{cases}
u_{tt}+(-\Delta)^{\theta}u_{t}-\Delta u=|u|^{p-1}u,
& (x,t)\in\Omega\times(0,T),\\[2pt]
u=0, & (x,t)\in\partial\Omega\times(0,T),\\[2pt]
u(x,0)=u_{0}(x),\quad u_{t}(x,0)=u_{1}(x), & x\in\Omega,
\end{cases}
\label{eq:main}
\end{equation}
where $(-\Delta)^{\theta}$ denotes the spectral fractional power of the
Dirichlet Laplacian and the exponent $p$ satisfies
\begin{equation}
p>1 \ \text{ if } n\in\{1,2\},
\qquad
1<p<\tfrac{n+2}{n-2} \ \text{ if } n\ge3.
\label{eq:Hp}
\end{equation}
When local well-posedness results are quoted we assume in addition that
$\partial\Omega$ is of class $C^{1,1}$; every functional inequality used in
this paper is valid for an arbitrary bounded open set, as explained in
Remark~\ref{rem:noboundary}. Both the solution of \eqref{eq:main} and its
maximal existence time depend on $\theta$; since $\theta$ is fixed we do not
display this dependence, except in Corollary~\ref{cor:uniform}, where a
bound valid simultaneously for all $\theta\in[0,1]$ is recorded and the
notation $T_{\max}(\theta)$ is used.

The parameter $\theta$ selects a dissipation mechanism. At $\theta=0$ the
damping is the frictional term $u_{t}$, a zeroth-order dissipation which
acts equally on all frequencies and models energy loss external to the
material, such as air resistance or dry friction at the supports of a
structure. At $\theta=1$ it is the Kelvin--Voigt, or strong, damping
$-\Delta u_{t}$, a second-order dissipation which acts most strongly on the
high frequencies and models the internal viscous stress of a viscoelastic
solid, lending the equation a parabolic character. Intermediate values of
$\theta$ interpolate between the two and correspond to the \emph{structural
damping} introduced by Chen and Triggiani \cite{ChenTriggiani1989,
ChenTriggiani1990} in the analysis of damped elastic systems: the underlying
linear semigroup is analytic for $\theta\in[\frac12,1]$ and of Gevrey class
for $\theta\in(0,\frac12)$. The scale therefore offers a one-parameter
family of models whose frequency dependence can be matched to a measured
damping law, rather than a choice between two extremes. Decay rates and
critical exponents for the corresponding semilinear Cauchy problem in
$\R^{n}$ have been studied in depth, see for instance
\cite{IkehataTodorovaYordanov2013,DAbbiccoReissig2014} and the references
there.

The superlinear source term $|u|^{p-1}u$ in \eqref{eq:main} represents a
self-reinforcing load which grows faster than the restoring force of the
medium, and the finite-time blow-up of the solution is the model's
expression of loss of structural integrity. Configurations of this kind
arise whenever slender, flexible components carry dynamic loads large enough
for the linear theory to fail, as in bridges, towers and aircraft wings,
and the engineering response is to introduce damping devices whose
dissipation depends nonlinearly on velocity; see
\cite{Ikhouane2007,Kerschen2006,Lazar2013,Spencer1997}. For such models the
quantity of practical interest is not only whether failure occurs but
\emph{when}. A lower bound for the blow-up time is a guaranteed interval on
which the model is still valid, and an upper bound is a guaranteed failure
time; together they bracket the lifespan. The question addressed in the
present paper is how that bracket depends on the strength $\theta$ of the
dissipation, and how far the dependence can be displayed explicitly.

For the two endpoint values of $\theta$ the qualitative picture is
classical. At $\theta=0$, finite-time blow-up under negative initial energy
goes back to the concavity method of Levine \cite{Levine1974} and
Kalantarov--Ladyzhenskaya \cite{KalantarovLadyzhenskaya1978}, and was later
extended to the whole range of initial energies below the depth $d$ of the
potential well of Payne--Sattinger \cite{PayneSattinger1975}, see
\cite{Ikehata1996,Vitillaro1999}. At $\theta=1$, the dichotomy between
global existence and finite-time blow-up was settled by Gazzola and
Squassina \cite{GazzolaSquassina2006}, and quantitative two-sided bounds for
the blow-up time in the strongly damped case were obtained in
\cite{BchatniaHamoudaKaabi2026}.

Explicit lower bounds for the blow-up time, in the spirit of Payne and
Schaefer \cite{PayneSchaefer2007}, have been derived for a variety of damped
hyperbolic problems: for the damped semilinear wave equation by Sun, Guo and
Gao \cite{SunGuoGao2014}, for two model wave equations by Zhou
\cite{Zhou2015}, for viscoelastic equations with nonlinear sources by Guo and
Liu \cite{GuoLiu2016}, and, together with upper bounds, for variable sources
by Sun, Ren and Gao \cite{SunRenGao2016}. In all these works the mechanism
is the same: a first-order differential inequality for a norm of the
solution, integrated from the initial time and combined with the blow-up
alternative. What varies is the class of admissible exponents, and that
class is dictated by the smoothing that the dissipation makes available.

The reference closest to the present paper is the work of Ding and Zhou
\cite{DingZhou2022} on a quasilinear wave equation with structural or strong
damping, where local well-posedness, global existence, asymptotic behaviour,
blow-up at subcritical, critical and arbitrarily high initial energy, and
upper and lower bounds for the blow-up time are all obtained; the proof of
global existence given there was subsequently corrected by the same authors
in \cite{DingZhou2023}. The exponent restriction \eqref{eq:plower} below,
which interpolates linearly between the two endpoint ranges, is of the same
nature as the restriction appearing in that analysis, and we make no claim
of novelty for it.

Our purpose is complementary to \cite{DingZhou2022}. It is to treat the
semilinear equation \eqref{eq:main} on the entire scale $\theta\in[0,1]$,
including the frictional endpoint $\theta=0$ and the Gevrey range
$\theta\in(0,\frac12)$, where the semigroup is no longer analytic and the
standard analytic-semigroup smoothing argument is unavailable, and to display
explicitly how $\theta$ enters both lifespan estimates. Concretely, the
contribution is fourfold. First, a single Levine-type functional, namely
\eqref{eq:M} below, produces the upper bound for every $\theta\in[0,1]$ at
once: it is built so that the fractional dissipation cancels identically
when the second derivative is computed, and consequently no property of
$(-\Delta)^{\theta}$ beyond self-adjointness and positivity is used. The
threshold parameter $\sigma_{\theta}$ of Theorem~\ref{thm:upper} is given by
a closed formula, and Corollary~\ref{cor:uniform} turns the bound into one
that holds simultaneously for all $\theta\in[0,1]$. Second, the lower bound
is obtained from the dissipation used quantitatively, and the admissible
exponent range \eqref{eq:plower} appears as a linear interpolation between
the frictional range $p\le\frac{n}{n-2}$ and the Kelvin--Voigt range
$p\le\frac{n+2}{n-2}$, with all constants identified as norms of two
explicit embeddings. Third, both theorems are proved for a class of energy
solutions specified in Definition~\ref{def:solution}, so that they are
insensitive to which local well-posedness theorem is invoked. This point
matters here, because for intermediate $\theta$ and exponents
$p>\frac{n}{n-2}$ the local theory in the energy space is delicate, as
discussed in Remark~\ref{rem:lwpdiscussion}. Fourth, because the proofs use
no property of $-\Delta$ beyond self-adjointness, positivity and compactness
of the resolvent, both bounds transfer verbatim to
$u_{tt}+A^{\theta}u_{t}+Au=|u|^{p-1}u$ for any operator $A$ of that type
whose form domain embeds in $L^{p+1}(\Omega)$. This covers the operators
that actually occur in structural models, in particular the linear
elasticity operator and the fourth-order plate and beam operators with
internal damping; see Remark~\ref{rem:abstract}. For $\theta=1$ the two
theorems reduce to estimates of the type obtained in
\cite{GazzolaSquassina2006,BchatniaHamoudaKaabi2026}, and for $\theta=0$ to
estimates of the type in
\cite{Levine1974,KalantarovLadyzhenskaya1978,SunGuoGao2014}.

The results below are obtained by analytical means only, through energy
estimates and differential inequalities. A numerical validation of the
corresponding two-sided bracket in the strongly damped case $\theta=1$,
using physics-informed neural networks, is reported in
\cite{BchatniaHamoudaKaabi2026}; extending such a computation across the
scale requires a discretisation of the nonlocal operator $A^{\theta}$, for
which a spectral Galerkin representation is the natural choice, and we
regard this as a separate problem which is not addressed here.

This paper is structured as follows. Section~\ref{sec:main} fixes the
notation, specifies the class of solutions for which we work, and states the
two main results, Theorems~\ref{thm:upper} and~\ref{thm:lower}, together
with their corollaries and a series of remarks.
Section~\ref{sec:auxiliary} collects the auxiliary results needed in the
proofs: the spectral inequalities for $A^{\theta/2}$, the embeddings behind
the constants appearing in Theorem~\ref{thm:lower}, the parameter ranges in
which energy solutions are known to exist, and two lemmas, one for each main
theorem. Section~\ref{sec:upper} establishes the upper bound and
Section~\ref{sec:lower} the lower bound. Section~\ref{sec:conclusion}
summarises the results, records the extension to elasticity and plate
operators, and lists the questions left open.

\section{Notation and main results}\label{sec:main}

We write $A:=-\Delta$ with Dirichlet boundary conditions, and we denote by
$(\lambda_{k},e_{k})_{k\ge1}$ its Dirichlet eigenpairs, with
$0<\lambda_{1}\le\lambda_{2}\le\cdots$ and $(e_{k})_{k\ge1}$ an orthonormal
basis of $L^{2}(\Omega)$. For $s\ge0$ the spectral fractional powers of $A$
are
\begin{equation}
A^{s}v=\sum_{k\ge1}\lambda_{k}^{s}(v,e_{k})e_{k},
\qquad
D(A^{s})=\Bigl\{v\in L^{2}(\Omega):
\sum_{k\ge1}\lambda_{k}^{2s}(v,e_{k})^{2}<\infty\Bigr\},
\label{eq:powers}
\end{equation}
so that $D(A^{1/2})=H_{0}^{1}(\Omega)$ with $\n{A^{1/2}v}=\n{\nabla v}$ and
$A^{0}=\Id$; in particular $H_{0}^{1}(\Omega)\subset D(A^{\theta/2})$ for
every $\theta\in[0,1]$, so that $A^{\theta/2}v$ is defined for every
$v\in H_{0}^{1}(\Omega)$. Written in the eigenbasis, the damping term of
\eqref{eq:main} acts on the $k$th mode as multiplication by
$\lambda_{k}^{\theta}$, so that $\theta$ is precisely the rate at which the
dissipation grows with the modal frequency. We denote by $(\cdot,\cdot)$ the
inner product of $L^{2}(\Omega)$, by $\langle\cdot,\cdot\rangle$ the duality
pairing between $H^{-1}(\Omega)$ and $H_{0}^{1}(\Omega)$, and by
$\n{\cdot}_{q}$ the norm of $L^{q}(\Omega)$; we abbreviate
$\n{\cdot}:=\n{\cdot}_{2}$.

The energy and Nehari functionals associated with \eqref{eq:main} are
\begin{equation}
E(t):=\tfrac12\n{u_{t}(t)}^{2}+\tfrac12\n{\nabla u(t)}^{2}
-\tfrac{1}{p+1}\n{u(t)}_{p+1}^{p+1},
\qquad
I(v):=\n{\nabla v}^{2}-\n{v}_{p+1}^{p+1},
\label{eq:energy}
\end{equation}
the depth of the potential well is
\begin{equation}
d:=\inf\bigl\{J(v):\, v\in H_{0}^{1}(\Omega)\setminus\{0\},\ I(v)=0\bigr\},
\qquad
J(v):=\tfrac12\n{\nabla v}^{2}-\tfrac{1}{p+1}\n{v}_{p+1}^{p+1},
\label{eq:depth}
\end{equation}
which under \eqref{eq:Hp} is positive and equal to
$d=\frac{p-1}{2(p+1)}S_{*}^{-\frac{2(p+1)}{p-1}}$, where
\begin{equation}
S_{*}:=\sup_{0\neq v\in H_{0}^{1}(\Omega)}
\frac{\n{v}_{p+1}}{\n{\nabla v}}
\label{eq:Sstar}
\end{equation}
is the best constant of the embedding $H_{0}^{1}(\Omega)\hookrightarrow
L^{p+1}(\Omega)$ \cite{PayneSattinger1975}. Finally we set
\begin{equation}
\Phi(t):=\n{u_{t}(t)}^{2}+\n{\nabla u(t)}^{2},
\label{eq:Phi}
\end{equation}
twice the mechanical energy of the linear part, and the quantity whose
divergence characterises blow-up.

Since both main results are a priori estimates, it is both cleaner and safer
to isolate the properties of the solution that they actually use.

\begin{definition}\label{def:solution}
Let $\theta\in[0,1]$, let $p$ satisfy \eqref{eq:Hp}, let
$(u_{0},u_{1})\in H_{0}^{1}(\Omega)\times L^{2}(\Omega)$ and let
$T\in(0,\infty]$. A function $u$ is an \emph{energy solution} of
\eqref{eq:main} on $[0,T)$ if
\begin{enumerate}
\item[(S1)] $u\in C\bigl([0,T);H_{0}^{1}(\Omega)\bigr)\cap
C^{1}\bigl([0,T);L^{2}(\Omega)\bigr)$ and
$u_{t}\in L^{2}_{\mathrm{loc}}\bigl([0,T);D(A^{\theta/2})\bigr)$;
\item[(S2)] $u(0)=u_{0}$, $u_{t}(0)=u_{1}$, and
$u_{tt}\in L^{2}_{\mathrm{loc}}\bigl([0,T);H^{-1}(\Omega)\bigr)$ with
$$
\langle u_{tt},v\rangle
+\bigl(A^{\theta/2}u_{t},A^{\theta/2}v\bigr)
+(\nabla u,\nabla v)
=\bigl\langle|u|^{p-1}u,v\bigr\rangle
$$
for a.e.\ $t\in(0,T)$ and all $v\in H_{0}^{1}(\Omega)$;
\item[(S3)] the energy inequality
$$
E(t)+\int_{0}^{t}\n{A^{\theta/2}u_{s}(s)}^{2}\ds\le E(0),
\qquad t\in[0,T),
$$
holds;
\item[(S4)] $T$ is maximal, in the sense that either $T=\infty$, or
$T<\infty$ and $\limsup_{t\to T^{-}}\Phi(t)=+\infty$.
\end{enumerate}
We then write $T=T_{\max}$, and we say that $u$ \emph{blows up at time
$T_{\max}$} when $T_{\max}<\infty$; in that case (S4) is the blow-up
alternative.
\end{definition}

All the terms in (S2) are well defined: $u\in H_{0}^{1}(\Omega)$ and
\eqref{eq:Hp} give $|u|^{p-1}u\in L^{(p+1)/p}(\Omega)\subset
H^{-1}(\Omega)$, while $A^{\theta/2}v$ makes sense for
$v\in H_{0}^{1}(\Omega)\subset D(A^{\theta/2})$. The quantity
$\int_{0}^{t}\n{A^{\theta/2}u_{s}}^{2}\ds$ appearing in (S3) is the energy
dissipated by the damping up to time $t$, so that (S3) states that the sum
of the stored and dissipated energy does not exceed its initial value. Only
this inequality is required, not the corresponding identity. This is what
limits of Galerkin approximations provide, and it is all that the proofs
below use. Parameter ranges in which energy solutions are known to exist
are recorded in Proposition~\ref{prop:local}, and the ranges in which we do
not claim existence are discussed in Remark~\ref{rem:lwpdiscussion}.

The first main result gives blow-up together with an upper bound for the
lifespan. Its hypotheses are the classical ones of the potential-well
method; what is specific to \eqref{eq:main} is the explicit appearance of
$\n{A^{\theta/2}u_{0}}^{2}$, and nothing else.

\begin{theorem}[Blow-up and upper bound of the blow-up time]
\label{thm:upper}
Let $\theta\in[0,1]$, assume \eqref{eq:Hp}, and let
$(u_{0},u_{1})\in H_{0}^{1}(\Omega)\times L^{2}(\Omega)$ satisfy
\begin{equation}
E(0)<d
\qquad\text{and}\qquad
I(u_{0})<0 .
\label{eq:datahyp}
\end{equation}
Let $u$ be an energy solution of \eqref{eq:main} on $[0,T_{\max})$. Fix
\begin{equation}
\beta\in\bigl(0,\,2(d-E(0))\bigr]
\label{eq:beta}
\end{equation}
and set
\begin{equation}
\sigma_{\theta}:=\max\Bigl\{0,\
\frac{1}{\beta}\Bigl[\frac{2}{p-1}\n{A^{\theta/2}u_{0}}^{2}
-(u_{0},u_{1})\Bigr]\Bigr\}.
\label{eq:sigmatheta}
\end{equation}
Then for every $\sigma>\sigma_{\theta}$ one has
\begin{equation}
\tfrac{p-1}{2}\bigl[(u_{0},u_{1})+\beta\sigma\bigr]
>\n{A^{\theta/2}u_{0}}^{2},
\label{eq:sigma}
\end{equation}
the solution blows up in finite time, and
\begin{equation}
T_{\max}\ \le\
\frac{\n{u_{0}}^{2}+\beta\sigma^{2}}
{\dfrac{p-1}{2}\bigl[(u_{0},u_{1})+\beta\sigma\bigr]
-\n{A^{\theta/2}u_{0}}^{2}} .
\label{eq:upperbound}
\end{equation}
\end{theorem}

The right-hand side of \eqref{eq:upperbound} depends on $\theta$ only
through $\n{A^{\theta/2}u_{0}}^{2}$, and that quantity is dominated,
uniformly in $\theta\in[0,1]$, by
\begin{equation}
N_{0}:=\max\bigl\{\n{u_{0}}^{2},\n{\nabla u_{0}}^{2}\bigr\},
\label{eq:N0}
\end{equation}
by the interpolation inequality \eqref{eq:interp} of
Lemma~\ref{lem:interp}. Replacing $\n{A^{\theta/2}u_{0}}^{2}$ by $N_{0}$
therefore weakens the bound slightly and removes $\theta$ from it
altogether, which is useful when the damping law of the material is known
only to lie somewhere on the scale.

\begin{corollary}[A bound uniform in $\theta$]\label{cor:uniform}
Assume \eqref{eq:Hp}, \eqref{eq:datahyp} and \eqref{eq:beta}, and let
$$
\sigma_{*}:=\max\Bigl\{0,\
\frac{1}{\beta}\Bigl[\frac{2}{p-1}N_{0}-(u_{0},u_{1})\Bigr]\Bigr\},
\qquad\text{so that }\ \sigma_{*}\ge\sigma_{\theta}
\ \text{ for every }\theta\in[0,1].
$$
Then for every $\sigma>\sigma_{*}$ and every $\theta\in[0,1]$, every energy
solution of \eqref{eq:main} blows up in finite time and
$$
T_{\max}(\theta)\ \le\
\frac{\n{u_{0}}^{2}+\beta\sigma^{2}}
{\frac{p-1}{2}\bigl[(u_{0},u_{1})+\beta\sigma\bigr]-N_{0}} .
$$
\end{corollary}

The second main result goes in the opposite direction. It requires no
assumption at all on the initial energy, and it is the place where the
dissipation is used rather than neutralised.

\begin{theorem}[Lower bound of the blow-up time]\label{thm:lower}
Let $\theta\in[0,1]$, assume \eqref{eq:Hp} and, if $n\ge3$, assume in
addition
\begin{equation}
1<p\le\frac{n+2\theta}{n-2}.
\label{eq:plower}
\end{equation}
Let $(u_{0},u_{1})\in H_{0}^{1}(\Omega)\times L^{2}(\Omega)$ with
$(u_{0},u_{1})\neq(0,0)$, and let $u$ be an energy solution of
\eqref{eq:main} on $[0,T_{\max})$. Then
\begin{equation}
T_{\max}\ \ge\ \frac{1}{(p-1)\,K_{p,\theta}\,\Phi(0)^{\,p-1}},
\qquad
\Phi(0)=\n{u_{1}}^{2}+\n{\nabla u_{0}}^{2},
\label{eq:lowerbound}
\end{equation}
where $K_{p,\theta}:=\frac12 S_{\theta}^{2}B_{p,\theta}^{2p}$ and, with
$q_{\theta}:=\frac{2n}{n-2\theta}$ and
$q_{\theta}':=\frac{2n}{n+2\theta}$,
\begin{equation}
S_{\theta}:=\sup_{0\neq v\in D(A^{\theta/2})}
\frac{\n{v}_{q_{\theta}}}{\n{A^{\theta/2}v}},
\qquad
B_{p,\theta}:=\sup_{0\neq v\in H_{0}^{1}(\Omega)}
\frac{\n{v}_{p q_{\theta}'}}{\n{\nabla v}} .
\label{eq:constants}
\end{equation}
For $n\le2$ the same conclusion holds for every $p>1$ upon taking
$q_{\theta}:=2$, in which case $S_{\theta}=\lambda_{1}^{-\theta/2}$, and
$q_{\theta}'=2$.
\end{theorem}

Combining the two theorems brackets the lifespan between two computable
quantities.

\begin{corollary}[Two-sided estimate]\label{cor:twosided}
Under the assumptions of Theorems~\ref{thm:upper} and~\ref{thm:lower}, and
with $\beta$ and $\sigma$ chosen as in Theorem~\ref{thm:upper},
$$
\frac{1}{(p-1)K_{p,\theta}\bigl(\n{u_{1}}^{2}+\n{\nabla u_{0}}^{2}\bigr)^{p-1}}
\ \le\ T_{\max}\ \le\
\frac{\n{u_{0}}^{2}+\beta\sigma^{2}}
{\frac{p-1}{2}\bigl[(u_{0},u_{1})+\beta\sigma\bigr]
-\n{A^{\theta/2}u_{0}}^{2}} .
$$
\end{corollary}

Both ends of the bracket are computable from the initial data alone, given
the first Dirichlet eigenvalue and the three embedding constants
\eqref{eq:Sstar} and \eqref{eq:constants}. The left-hand end is a time
before which the model is guaranteed to remain valid, and the right-hand end
a time by which failure has certainly occurred; the width of the interval
measures how much the two mechanisms of proof leave undetermined.

\begin{remark}\label{rem:comments}
(i) If $E(0)<0$, then \eqref{eq:datahyp} holds automatically. Indeed
$E(0)<0$ gives
$\frac{1}{p+1}\n{u_{0}}_{p+1}^{p+1}>\frac12\n{u_{1}}^{2}
+\frac12\n{\nabla u_{0}}^{2}\ge\frac12\n{\nabla u_{0}}^{2}$, so that
$\n{u_{0}}_{p+1}^{p+1}>\frac{p+1}{2}\n{\nabla u_{0}}^{2}
\ge\n{\nabla u_{0}}^{2}$ because $p\ge1$; hence $I(u_{0})<0$, and
$E(0)<0<d$. Mechanically, this is the regime in which the work available
from the load already exceeds the stored elastic energy at the initial
instant.

(ii) By \eqref{eq:interp} one has $\n{A^{\theta/2}u_{0}}^{2}
\le\n{u_{0}}^{2(1-\theta)}\n{\nabla u_{0}}^{2\theta}$, so the right-hand
side of \eqref{eq:upperbound} is fully explicit in $\theta$.

(iii) The restriction \eqref{eq:plower} interpolates linearly between the
classical ranges $p\le\frac{n}{n-2}$ for frictional damping and
$p\le\frac{n+2}{n-2}$ for Kelvin--Voigt damping. The smoothing strength of
the structural damping is thus quantitatively visible in the class of
nonlinearities for which an explicit lower bound of the lifespan is
available: the stiffer the internal dissipation, the more strongly
superlinear the load may be while a guaranteed safe interval can still be
computed. We are not aware of an explicit estimate of the form
\eqref{eq:lowerbound} in the remaining range
$\frac{n+2\theta}{n-2}<p<\frac{n+2}{n-2}$, $\theta\in[0,1)$.

(iv) We emphasise what \eqref{eq:plower} does and does not say. It widens
the set of exponents for which a lower bound is available as $\theta$
increases. It does not assert that the number on the right of
\eqref{eq:lowerbound} increases with $\theta$: the constant $K_{p,\theta}$
is a product of the two embedding norms \eqref{eq:constants}, whose
monotonicity in $\theta$ we do not establish.

(v) In \eqref{eq:upperbound} the parameter $\beta$ is fixed first, inside
$(0,2(d-E(0))]$, and only then is $\sigma$ chosen; the bound is therefore a
one-parameter family, indexed by $\sigma>\sigma_{\theta}$ with $\beta$
prescribed, rather than the outcome of a joint optimisation in
$(\beta,\sigma)$. Nothing prevents minimising the right-hand side of
\eqref{eq:upperbound} numerically over the admissible pairs, which is what
one would do in practice to sharpen a predicted failure time.
\end{remark}

\section{Auxiliary results}\label{sec:auxiliary}

This section collects the material needed in the two proofs, in the order in
which it will be used: two spectral inequalities, the embeddings behind the
constants \eqref{eq:constants}, the parameter ranges in which energy
solutions exist, and finally one lemma for each main theorem.

\begin{lemma}\label{lem:interp}
For every $\theta\in[0,1]$ and every $v\in H_{0}^{1}(\Omega)$,
\begin{equation}
\n{A^{\theta/2}v}^{2}\le\n{A^{1/2}v}^{2\theta}\n{v}^{2(1-\theta)}
\le\max\bigl\{\n{v}^{2},\n{\nabla v}^{2}\bigr\},
\label{eq:interp}
\end{equation}
and
\begin{equation}
\n{A^{\theta/2}v}\le\lambda_{1}^{-(1-\theta)/2}\n{\nabla v},
\qquad
\n{v}\le\lambda_{1}^{-\theta/2}\n{A^{\theta/2}v}.
\label{eq:poincare}
\end{equation}
\end{lemma}

\begin{proof}
Write $v=\sum_{k}v_{k}e_{k}$. For $\theta\in(0,1)$, H\"older's inequality
with exponents $\frac1\theta$ and $\frac{1}{1-\theta}$ applied to the
spectral sum \eqref{eq:powers} gives
$$
\n{A^{\theta/2}v}^{2}=\sum_{k}\lambda_{k}^{\theta}v_{k}^{2}
=\sum_{k}(\lambda_{k}v_{k}^{2})^{\theta}(v_{k}^{2})^{1-\theta}
\le\Bigl(\sum_{k}\lambda_{k}v_{k}^{2}\Bigr)^{\theta}
\Bigl(\sum_{k}v_{k}^{2}\Bigr)^{1-\theta}
=\n{A^{1/2}v}^{2\theta}\n{v}^{2(1-\theta)},
$$
the cases $\theta\in\{0,1\}$ being trivial; the second inequality in
\eqref{eq:interp} is immediate. For \eqref{eq:poincare} use
$\lambda_{k}^{\theta}\le\lambda_{1}^{-(1-\theta)}\lambda_{k}$ and
$1\le\lambda_{1}^{-\theta}\lambda_{k}^{\theta}$ respectively.
\end{proof}

Each of these inequalities is used more than once. The first makes the
$\theta$-dependence of Theorem~\ref{thm:upper} visible and, through the bound
by $N_{0}$, yields Corollary~\ref{cor:uniform}. The first inequality in
\eqref{eq:poincare} is what shows, in the last step of the proof of
Theorem~\ref{thm:upper}, that the functional \eqref{eq:M} stays bounded as
long as the solution does; it also gives the boundedness of the bilinear form
used in Proposition~\ref{prop:local}. The second inequality in
\eqref{eq:poincare} says exactly that
$D(A^{\theta/2})\hookrightarrow L^{2}(\Omega)$ with norm
$\lambda_{1}^{-\theta/2}$, which is the value of $S_{\theta}$ recorded for
$n\le2$ in Theorem~\ref{thm:lower}.

For $n\ge3$ the constant $S_{\theta}$ of \eqref{eq:constants} comes from a
fractional embedding. The next remark records that it is finite and that it
requires no regularity of the boundary, which matters for applications, since
the domains occurring in structural models are rarely smooth.

\begin{remark}\label{rem:noboundary}
For $n\ge3$ and $0\le\theta\le1<\frac{n}{2}$ we use
\begin{equation}
D(A^{\theta/2})\hookrightarrow L^{q_{\theta}}(\Omega),
\qquad q_{\theta}=\frac{2n}{n-2\theta},
\label{eq:embed}
\end{equation}
whose norm is the quantity $S_{\theta}$ of \eqref{eq:constants}. No
regularity of $\partial\Omega$ is needed for \eqref{eq:embed}. Indeed the
Dirichlet heat semigroup on an arbitrary open set is dominated by the
Gaussian one, so that
$\n{e^{-tA}}_{L^{1}\to L^{\infty}}\le(4\pi t)^{-n/2}$, and by the
equivalence between such ultracontractivity bounds and Sobolev inequalities
for fractional powers of the generator \cite{Varopoulos1985,Davies1989},
$A^{-\theta/2}$ is bounded from $L^{2}(\Omega)$ to $L^{q}(\Omega)$ with
$\frac1q=\frac12-\frac\theta n$, that is $q=q_{\theta}$, with a constant
depending only on $n$ and $\theta$. Likewise
$H_{0}^{1}(\Omega)\hookrightarrow L^{r}(\Omega)$ for
$2\le r\le\frac{2n}{n-2}$ holds for any open $\Omega$ by extension by zero,
so that $B_{p,\theta}<\infty$; note that
\begin{equation}
p\,q_{\theta}'\le\frac{2n}{n-2}
\quad\Longleftrightarrow\quad
p\le\frac{n+2\theta}{n-2},
\label{eq:exprange}
\end{equation}
which is precisely \eqref{eq:plower}: the exponent condition of
Theorem~\ref{thm:lower} amounts to the requirement that $|u|^{p-1}u$ be
integrable against $u_{t}$ in the duality
$L^{q_{\theta}'}$--$L^{q_{\theta}}$, that is, that the power supplied by the
load be finite in the norm controlled by the damping. Boundary regularity is
used only in Proposition~\ref{prop:local}. We never identify
$D(A^{\theta/2})$ with a fractional Sobolev space; for smooth boundary such
identifications are available, but they require care, in particular at the
borderline $\theta=\frac12$, where Lions--Magenes spaces intervene
\cite{Fujiwara1967,LionsMagenes1972}.
\end{remark}

Theorems~\ref{thm:upper} and~\ref{thm:lower} apply in every parameter range
in which a local energy solution satisfying (S1)--(S4) is supplied by an
available well-posedness theory. We record two such situations; outside them
we make no claim.

\begin{proposition}\label{prop:local}
Assume $\partial\Omega\in C^{1,1}$ and
$(u_{0},u_{1})\in H_{0}^{1}(\Omega)\times L^{2}(\Omega)$.
\begin{enumerate}
\item[(i)] Let $\theta\in[0,1]$ and suppose that $n\le2$ and $p>1$, or that
$n\ge3$ and $1<p\le\frac{n}{n-2}$. Then \eqref{eq:main} has a unique
maximal energy solution in the sense of Definition~\ref{def:solution}, and
(S3) holds with equality.
\item[(ii)] Let $\theta=1$ and let $p$ satisfy \eqref{eq:Hp}. Then
\eqref{eq:main} has a maximal energy solution; this is
\cite{GazzolaSquassina2006}.
\end{enumerate}
\end{proposition}

\begin{proof}
Only (i) requires proof, and we divide it into three steps.

\emph{Step 1: the linear operator.} Set
$\mathcal{H}:=H_{0}^{1}(\Omega)\times L^{2}(\Omega)$, normed by
$\n{(u,v)}_{\mathcal{H}}^{2}=\n{\nabla u}^{2}+\n{v}^{2}$. For
$u,v\in H_{0}^{1}(\Omega)$ we understand $Au$ and $A^{\theta}v$ as elements
of $H^{-1}(\Omega)$ through the closed forms
\begin{equation}
\langle Au,w\rangle:=(\nabla u,\nabla w),
\qquad
\langle A^{\theta}v,w\rangle:=\bigl(A^{\theta/2}v,A^{\theta/2}w\bigr),
\qquad w\in H_{0}^{1}(\Omega),
\label{eq:forms}
\end{equation}
which is legitimate because $H_{0}^{1}(\Omega)\subset D(A^{\theta/2})$ and,
by \eqref{eq:poincare},
$|\langle A^{\theta}v,w\rangle|\le\lambda_{1}^{-(1-\theta)}
\n{\nabla v}\n{\nabla w}$. We then put
$$
\mathcal{A}(u,v):=(v,-Au-A^{\theta}v),
\qquad
D(\mathcal{A}):=\bigl\{(u,v)\in H_{0}^{1}(\Omega)\times H_{0}^{1}(\Omega):
Au+A^{\theta}v\in L^{2}(\Omega)\bigr\},
$$
the membership in $D(\mathcal{A})$ meaning that the functional
$Au+A^{\theta}v\in H^{-1}(\Omega)$ is represented by an element of
$L^{2}(\Omega)$; only the sum is required to have this property, not the two
terms separately. Since $D(A)\times D(A)\subset D(\mathcal{A})$ and $D(A)$
is dense in both $H_{0}^{1}(\Omega)$ and $L^{2}(\Omega)$, the operator
$\mathcal{A}$ is densely defined. It is dissipative: for
$(u,v)\in D(\mathcal{A})$, using \eqref{eq:forms},
\begin{align*}
\bigl\langle\mathcal{A}(u,v),(u,v)\bigr\rangle_{\mathcal{H}}
&=(\nabla v,\nabla u)-\bigl(Au+A^{\theta}v,\,v\bigr)\\
&=(\nabla v,\nabla u)-(\nabla u,\nabla v)-\n{A^{\theta/2}v}^{2}
=-\n{A^{\theta/2}v}^{2}\le0 ,
\end{align*}
the right-hand side being the instantaneous rate of energy dissipation.

\emph{Step 2: the range condition.} We claim that
$\Ran(\lambda\Id-\mathcal{A})=\mathcal{H}$ for every $\lambda>0$. Let
$(f,g)\in\mathcal{H}$. Solving $(\lambda\Id-\mathcal{A})(u,v)=(f,g)$ means
\begin{equation}
\lambda u-v=f,
\qquad
\lambda v+Au+A^{\theta}v=g .
\label{eq:resolvent}
\end{equation}
Eliminating $v=\lambda u-f$ leads to the variational problem
$$
\mathcal{B}(u,w)=\mathcal{L}(w)\quad\text{for all }w\in H_{0}^{1}(\Omega),
$$
where
$$
\mathcal{B}(u,w):=\lambda^{2}(u,w)+(\nabla u,\nabla w)
+\lambda\bigl(A^{\theta/2}u,A^{\theta/2}w\bigr),
\qquad
\mathcal{L}(w):=(g+\lambda f,w)
+\bigl(A^{\theta/2}f,A^{\theta/2}w\bigr).
$$
The form $\mathcal{B}$ is bilinear and bounded on
$H_{0}^{1}(\Omega)\times H_{0}^{1}(\Omega)$ by \eqref{eq:poincare} and the
Poincar\'e inequality, and it is coercive, since
$\mathcal{B}(u,u)\ge\n{\nabla u}^{2}$. The functional $\mathcal{L}$ is
bounded on $H_{0}^{1}(\Omega)$, because $g\in L^{2}(\Omega)$,
$f\in H_{0}^{1}(\Omega)\subset L^{2}(\Omega)$ and
$|(A^{\theta/2}f,A^{\theta/2}w)|\le\lambda_{1}^{-(1-\theta)}
\n{\nabla f}\n{\nabla w}$. By the Lax--Milgram theorem there is a unique
$u\in H_{0}^{1}(\Omega)$ with $\mathcal{B}(u,\cdot)=\mathcal{L}$. Put
$v:=\lambda u-f\in H_{0}^{1}(\Omega)$. Then, for every
$w\in H_{0}^{1}(\Omega)$,
\begin{align*}
\bigl\langle Au+A^{\theta}v,w\bigr\rangle
&=(\nabla u,\nabla w)+\lambda\bigl(A^{\theta/2}u,A^{\theta/2}w\bigr)
-\bigl(A^{\theta/2}f,A^{\theta/2}w\bigr)\\
&=\mathcal{B}(u,w)-\lambda^{2}(u,w)
-\bigl(A^{\theta/2}f,A^{\theta/2}w\bigr)
=(g,w)-\lambda(\lambda u-f,w)=(g-\lambda v,w).
\end{align*}
Hence $Au+A^{\theta}v=g-\lambda v$, which lies in $L^{2}(\Omega)$ because
$g\in L^{2}(\Omega)$ and $v\in H_{0}^{1}(\Omega)$. Therefore
$(u,v)\in D(\mathcal{A})$ and \eqref{eq:resolvent} holds, which proves the
claim. By the Lumer--Phillips theorem \cite[Ch.~1]{Pazy1983},
\cite[Ch.~2]{CazenaveHaraux1998}, $\mathcal{A}$ generates a $C_{0}$
semigroup of contractions on $\mathcal{H}$.

\emph{Step 3: the nonlinearity and the regularity of the solution.} Write
\eqref{eq:main} as $U'=\mathcal{A}U+F(U)$ with $U=(u,u_{t})$ and
$F(u,v):=(0,|u|^{p-1}u)$. The map $F$ is Lipschitz on bounded subsets of
$\mathcal{H}$. Indeed, from
$\bigl||a|^{p-1}a-|b|^{p-1}b\bigr|\le
C_{p}\bigl(|a|^{p-1}+|b|^{p-1}\bigr)|a-b|$ and H\"older's inequality with
exponents $n$ and $\frac{2n}{n-2}$, whose reciprocals sum to $\frac12$,
$$
\n{|u|^{p-1}u-|w|^{p-1}w}
\le C_{p}\bigl\|\,|u|^{p-1}+|w|^{p-1}\bigr\|_{n}\,
\n{u-w}_{\frac{2n}{n-2}},
$$
and $\bigl\||u|^{p-1}\bigr\|_{n}=\n{u}_{n(p-1)}^{p-1}$ is controlled by
$\n{\nabla u}^{p-1}$ precisely when $n(p-1)\le\frac{2n}{n-2}$, that is when
$p\le\frac{n}{n-2}$. For $n\le2$ the same computation works with the pair of
exponents $(4,4)$ in place of $\bigl(n,\frac{2n}{n-2}\bigr)$, since
$H_{0}^{1}(\Omega)\hookrightarrow L^{r}(\Omega)$ for every $r<\infty$ when
$n=2$ and $H_{0}^{1}(\Omega)\hookrightarrow L^{\infty}(\Omega)$ when $n=1$.

Hence $F:\mathcal{H}\to\mathcal{H}$ is locally Lipschitz, and the standard
theory of semilinear equations with locally Lipschitz nonlinearities
\cite[Ch.~4]{CazenaveHaraux1998}, \cite[Ch.~6]{Pazy1983} provides a unique
maximal mild solution $U=(u,u_{t})\in C([0,T_{\max});\mathcal{H})$, that is
$$
u\in C\bigl([0,T_{\max});H_{0}^{1}(\Omega)\bigr),
\qquad
u_{t}\in C\bigl([0,T_{\max});L^{2}(\Omega)\bigr),
$$
together with the continuation criterion: if $T_{\max}<\infty$ then
$\limsup_{t\to T_{\max}^{-}}\n{U(t)}_{\mathcal{H}}=\infty$. Since
$\n{U(t)}_{\mathcal{H}}^{2}=\Phi(t)$, this is precisely (S4). For initial
data in $D(\mathcal{A})$ the solution is strong and the energy identity
follows by testing the equation with $u_{t}$; approximation and continuous
dependence on the data extend it to general data in $\mathcal{H}$, so that
(S3) holds with equality. Consequently
$\int_{0}^{t}\n{A^{\theta/2}u_{s}}^{2}\ds=E(0)-E(t)<\infty$ on compact
subintervals, that is $u_{t}\in L^{2}_{\mathrm{loc}}
\bigl([0,T_{\max});D(A^{\theta/2})\bigr)$, which completes (S1); and the
equation itself, together with $Au\in C([0,T_{\max});H^{-1}(\Omega))$,
$A^{\theta}u_{t}\in L^{2}_{\mathrm{loc}}([0,T_{\max});H^{-1}(\Omega))$ and
$|u|^{p-1}u\in C([0,T_{\max});H^{-1}(\Omega))$, gives
$u_{tt}\in L^{2}_{\mathrm{loc}}([0,T_{\max});H^{-1}(\Omega))$ and hence
(S2).
\end{proof}

\begin{remark}\label{rem:lwpdiscussion}
For $n\ge3$ and $\frac{n}{n-2}<p<\frac{n+2}{n-2}$, the case
$\theta\in[\frac12,1)$ is not covered by Proposition~\ref{prop:local}, and
we do not claim it. The semigroup generated by $\mathcal{A}$ is analytic in
that range \cite{ChenTriggiani1989}, and the associated smoothing makes a
corresponding local well-posedness statement plausible; results of this type
are available for related quasilinear models \cite{DingZhou2022}. Carrying
this out would require specifying the fractional-domain smoothing estimate,
the fixed-point space and the time-integrability exponent, and re-deriving
the continuation criterion, and we prefer not to assert it without proof.
The case $\theta\in[0,\frac12)$ is harder still: the smoothing furnished by
$A^{\theta}u_{t}$ is then of Gevrey and not of analytic type
\cite{ChenTriggiani1990}, and it disappears at $\theta=0$, where uniqueness
of energy-class solutions in the subcritical range rests instead on
Strichartz estimates, delicate on a bounded domain even for smooth boundary
\cite{BurqLebeauPlanchon2008}. Existence of some solution satisfying
(S1)--(S3) can be obtained in wide generality by Faedo--Galerkin
approximation \cite{Lions1969,Ikehata1996,Vitillaro1999,DingZhou2022}, the
energy inequality (S3) being preserved in the limit by weak lower
semicontinuity of the norms on its left-hand side; what such an argument
does not by itself deliver is uniqueness or the continuation criterion (S4).
This is why Definition~\ref{def:solution} isolates the properties actually
used: the two main theorems hold for any solution in that class, and
presuppose neither uniqueness nor the energy identity as an equality. In
particular, the lifespan bracket of Corollary~\ref{cor:twosided} is
available for any solution produced by a Galerkin or finite element
discretisation whose limit obeys the energy inequality, which is the
situation met in practice.
\end{remark}

It remains to prepare the two proofs. Each rests on exactly one lemma, and
we give them in the order in which the theorems are proved. The first lemma
prepares Theorem~\ref{thm:upper}: it says that the hypotheses
\eqref{eq:datahyp} propagate in time, which is what turns the concavity
computation of Section~\ref{sec:upper} into a strict differential
inequality.

\begin{lemma}[Invariance of the unstable set]\label{lem:invariance}
Assume \eqref{eq:Hp} and \eqref{eq:datahyp}, and let $u$ be an energy
solution on $[0,T_{\max})$. Then, for all $t\in[0,T_{\max})$,
\begin{equation}
I(u(t))<0
\qquad\text{and}\qquad
\n{\nabla u(t)}^{2}>\frac{2(p+1)}{p-1}\,d .
\label{eq:invariance}
\end{equation}
\end{lemma}

\begin{proof}
By (S3), $J(u(t))\le E(t)\le E(0)<d$ for all $t\in[0,T_{\max})$. Let
$$
U:=\{t\in[0,T_{\max}):\ I(u(t))<0\}.
$$
The map $t\mapsto I(u(t))$ is continuous, by (S1) and the embedding
$H_{0}^{1}(\Omega)\hookrightarrow L^{p+1}(\Omega)$; hence $U$ is open in
$[0,T_{\max})$, and $0\in U$ by \eqref{eq:datahyp}.

We first note that
\begin{equation}
\n{\nabla u(t)}^{2}>S_{*}^{-\frac{2(p+1)}{p-1}}=\frac{2(p+1)}{p-1}\,d,
\qquad t\in U .
\label{eq:gradbelow}
\end{equation}
Indeed, for $t\in U$ we have $\n{\nabla u(t)}^{2}<\n{u(t)}_{p+1}^{p+1}\le
S_{*}^{p+1}\n{\nabla u(t)}^{p+1}$ by \eqref{eq:Sstar}, and
$\n{\nabla u(t)}>0$, since otherwise $I(u(t))=0$; so we may divide by
$\n{\nabla u(t)}^{2}$ and use $p>1$.

Next, $U$ is closed in $[0,T_{\max})$. Let $t_{j}\in U$ with
$t_{j}\to t\in[0,T_{\max})$. Passing to the limit in \eqref{eq:gradbelow}
gives $\n{\nabla u(t)}^{2}\ge\frac{2(p+1)}{p-1}d>0$, so $u(t)\neq0$, and
$I(u(t))\le0$ by continuity. Were $I(u(t))=0$, then $u(t)\neq0$ and
\eqref{eq:depth} would force $J(u(t))\ge d$, incompatible with
$J(u(t))<d$. Hence $I(u(t))<0$, that is $t\in U$.

Since $[0,T_{\max})$ is connected and $U$ is a nonempty subset of it which
is both open and closed, $U=[0,T_{\max})$. Together with
\eqref{eq:gradbelow} this is \eqref{eq:invariance}.
\end{proof}

The second lemma prepares Theorem~\ref{thm:lower}: it converts the energy
inequality (S3) into an integral inequality for $\Phi$. Two features of the
statement should be pointed out. First, it is stated in integrated form,
which is what allows the proof in Section~\ref{sec:lower} to avoid
differentiating $\n{\nabla u(t)}^{2}$, an operation not licensed by (S1)
when $\theta<1$. Second, the exponent condition \eqref{eq:plower}, imposed
in Theorem~\ref{thm:lower} because the estimate of Section~\ref{sec:lower}
needs it, turns out to be exactly what legitimises the differentiation of
$t\mapsto\n{u(t)}_{p+1}^{p+1}$ here.

\begin{lemma}[An integral inequality for $\Phi$]\label{lem:Phiineq}
Let $u$ be an energy solution on $[0,T)$ and assume \eqref{eq:plower} if
$n\ge3$, with no extra assumption if $n\le2$. Then, for every
$t\in[0,T)$,
\begin{equation}
\Phi(t)+2\int_{0}^{t}\n{A^{\theta/2}u_{s}}^{2}\ds
\ \le\ \Phi(0)+2\int_{0}^{t}\int_{\Omega}|u|^{p-1}u\,u_{s}\dx\ds ,
\label{eq:S4}
\end{equation}
all the integrals being absolutely convergent.
\end{lemma}

\begin{proof}
Let $t\in[0,T)$ and put $q:=q_{\theta}$ and $q':=q_{\theta}'$ when $n\ge3$,
and $q=q'=2$ when $n\le2$. By (S1) and \eqref{eq:embed},
\begin{equation}
u\in L^{\infty}\bigl(0,t;L^{pq'}(\Omega)\bigr),
\qquad
u_{s}\in L^{2}\bigl(0,t;L^{q}(\Omega)\bigr),
\label{eq:twobounds}
\end{equation}
where \eqref{eq:exprange} and \eqref{eq:plower} are used for the first
membership. Hence, by H\"older's inequality in $x$ and then in $s$,
\begin{equation}
\int_{0}^{t}\Bigl|\int_{\Omega}|u|^{p-1}u\,u_{s}\dx\Bigr|\ds
\le\int_{0}^{t}\n{u}_{pq'}^{p}\,\n{u_{s}}_{q}\ds<\infty .
\label{eq:absconv}
\end{equation}
Under the bounds \eqref{eq:twobounds}, the standard Steklov-averaging chain
rule gives that $s\mapsto\n{u(s)}_{p+1}^{p+1}$ is absolutely continuous on
$[0,t]$ and
\begin{equation}
\n{u(t)}_{p+1}^{p+1}-\n{u_{0}}_{p+1}^{p+1}
=(p+1)\int_{0}^{t}\int_{\Omega}|u|^{p-1}u\,u_{s}\dx\ds .
\label{eq:chainrule}
\end{equation}
Indeed, \eqref{eq:chainrule} is elementary for the Steklov averages of $u$
in time, and one passes to the limit using \eqref{eq:absconv}. Finally, by
\eqref{eq:energy} and \eqref{eq:Phi},
$$
E(t)=\tfrac12\Phi(t)-\tfrac{1}{p+1}\n{u(t)}_{p+1}^{p+1},
$$
so that (S3) reads
$$
\tfrac12\Phi(t)+\int_{0}^{t}\n{A^{\theta/2}u_{s}}^{2}\ds
\le\tfrac12\Phi(0)
+\tfrac{1}{p+1}\Bigl[\n{u(t)}_{p+1}^{p+1}-\n{u_{0}}_{p+1}^{p+1}\Bigr],
$$
and \eqref{eq:S4} follows from \eqref{eq:chainrule}.
\end{proof}

With these two lemmas the proofs can be given. They treat the damping in
opposite ways. In Section~\ref{sec:upper} it does not enter the estimate at
all: the functional is designed so that it cancels, and
Lemma~\ref{lem:invariance} supplies the strict positivity that drives the
concavity argument. In Section~\ref{sec:lower} it is used quantitatively: it
absorbs the nonlinear term through Young's inequality, and
Lemma~\ref{lem:Phiineq} is the form in which it enters. The two sections are
independent of each other and may be read in either order.

\section{Upper bound for the blow-up time}\label{sec:upper}

The strategy is Levine's concavity method: one exhibits a positive
functional $M$ such that $M^{-a}$ is concave for some $a>0$, so that
$M^{-a}$ must reach zero in finite time, forcing $M$ to diverge. The
difficulty with \eqref{eq:main} is that a naive choice of $M$ leaves an
uncontrolled mixed term $(A^{\theta/2}u,A^{\theta/2}u_{t})$ of indefinite
sign. The functional \eqref{eq:M} below is built so that this term cancels
identically, and that is why the argument runs unchanged for every
$\theta\in[0,1]$.

\begin{proof}[Proof of Theorem~\ref{thm:upper}]
Fix $\beta$ as in \eqref{eq:beta} and $\sigma>\sigma_{\theta}$. By the
definition \eqref{eq:sigmatheta} of $\sigma_{\theta}$ we have
$\beta\sigma>\frac{2}{p-1}\n{A^{\theta/2}u_{0}}^{2}-(u_{0},u_{1})$, which is
\eqref{eq:sigma}. Consequently the number
\begin{equation}
T_{1}:=\frac{\n{u_{0}}^{2}+\beta\sigma^{2}}
{\frac{p-1}{2}\bigl[(u_{0},u_{1})+\beta\sigma\bigr]
-\n{A^{\theta/2}u_{0}}^{2}}
\label{eq:T1}
\end{equation}
is well defined and positive. Set $T_{2}:=\min\{T_{1},T_{\max}\}$ and
define, for $t\in[0,T_{2})$,
\begin{equation}
M(t):=\n{u(t)}^{2}
+\int_{0}^{t}\n{A^{\theta/2}u(s)}^{2}\ds
+(T_{1}-t)\n{A^{\theta/2}u_{0}}^{2}
+\beta(t+\sigma)^{2} .
\label{eq:M}
\end{equation}

\emph{Step 1: differentiation of $M$.} Every summand in \eqref{eq:M} is
nonnegative on $[0,T_{2})\subset[0,T_{1})$ and the last one is positive, so
$M>0$ there. By (S1), $t\mapsto\n{u(t)}^{2}$ is $C^{1}$ with derivative
$2(u,u_{t})$, the integral term is $C^{1}$ with derivative
$\n{A^{\theta/2}u(t)}^{2}$, and $t\mapsto\n{A^{\theta/2}u(t)}^{2}$ belongs
to $W^{1,1}_{\mathrm{loc}}$ with derivative
$2\bigl(A^{\theta/2}u,A^{\theta/2}u_{t}\bigr)$, because
$u\in H^{1}_{\mathrm{loc}}\bigl([0,T_{\max});D(A^{\theta/2})\bigr)$ by (S1).
Hence $M\in C^{1}$,
\begin{equation}
M'(t)=2(u,u_{t})
+2\int_{0}^{t}\bigl(A^{\theta/2}u,A^{\theta/2}u_{s}\bigr)\ds
+2\beta(t+\sigma),
\label{eq:Mprime}
\end{equation}
$M'$ is locally absolutely continuous, and for a.e.\ $t\in(0,T_{2})$
\begin{equation}
M''(t)=2\n{u_{t}}^{2}+2\langle u_{tt},u\rangle
+2\bigl(A^{\theta/2}u,A^{\theta/2}u_{t}\bigr)+2\beta .
\label{eq:Mpre}
\end{equation}

\emph{Step 2: the dissipation cancels.} Taking $v=u(t)$ in (S2), and using
$\bigl\langle|u|^{p-1}u,u\bigr\rangle=\n{u}_{p+1}^{p+1}$,
$$
\langle u_{tt},u\rangle
=-\bigl(A^{\theta/2}u_{t},A^{\theta/2}u\bigr)
-\n{\nabla u}^{2}+\n{u}_{p+1}^{p+1}
\qquad\text{for a.e. }t,
$$
and substituting into \eqref{eq:Mpre} the two terms
$\pm2\bigl(A^{\theta/2}u,A^{\theta/2}u_{t}\bigr)$ cancel exactly:
\begin{equation}
M''(t)=2\n{u_{t}}^{2}-2\n{\nabla u}^{2}+2\n{u}_{p+1}^{p+1}+2\beta .
\label{eq:Msecond}
\end{equation}
This identical cancellation, valid for every $\theta\in[0,1]$, is the purpose
of the second and third terms in \eqref{eq:M}, and is the reason why the
bound \eqref{eq:upperbound} has the same form along the entire damping
scale.

\emph{Step 3: a differential inequality.} By \eqref{eq:energy} and (S3),
$$
\n{u}_{p+1}^{p+1}
=(p+1)\Bigl[\tfrac12\Phi(t)-E(t)\Bigr]
\ \ge\ (p+1)\Bigl[\tfrac12\Phi(t)-E(0)
+\int_{0}^{t}\n{A^{\theta/2}u_{s}}^{2}\ds\Bigr],
$$
so that \eqref{eq:Msecond} gives
\begin{equation}
M''(t)\ \ge\ (p+3)\n{u_{t}}^{2}+(p-1)\n{\nabla u}^{2}
+2(p+1)\int_{0}^{t}\n{A^{\theta/2}u_{s}}^{2}\ds
-2(p+1)E(0)+2\beta .
\label{eq:Msecond2}
\end{equation}
By Lemma~\ref{lem:invariance}, $(p-1)\n{\nabla u(t)}^{2}>2(p+1)d$ on
$[0,T_{\max})$, hence on $[0,T_{2})$. Using in addition $2(p+1)\ge p+3$,
valid for $p\ge1$, on the dissipation integral, and $\beta\le2(d-E(0))$ from
\eqref{eq:beta}, we obtain
$2(p+1)(d-E(0))+2\beta\ge(p+1)\beta+2\beta=(p+3)\beta$ and therefore
\begin{equation}
M''(t)\ >\ (p+3)\,Q(t),
\qquad
Q(t):=\n{u_{t}}^{2}
+\int_{0}^{t}\n{A^{\theta/2}u_{s}}^{2}\ds+\beta ,
\label{eq:MQ}
\end{equation}
for a.e.\ $t\in(0,T_{2})$.

\emph{Step 4: concavity.} Applying the Cauchy--Schwarz inequality in
$L^{2}(\Omega)$ to each pairing in \eqref{eq:Mprime}, and then the
Cauchy--Schwarz inequality in $\R^{3}$ to the resulting three products,
\begin{align*}
\bigl(M'(t)\bigr)^{2}
&\le4\Bigl[\n{u}\,\n{u_{t}}
+\Bigl(\int_{0}^{t}\n{A^{\theta/2}u}^{2}\ds\Bigr)^{\!1/2}
\Bigl(\int_{0}^{t}\n{A^{\theta/2}u_{s}}^{2}\ds\Bigr)^{\!1/2}
+\sqrt{\beta}\,(t+\sigma)\cdot\sqrt{\beta}\Bigr]^{2}\\
&\le4\Bigl[\n{u}^{2}+\int_{0}^{t}\n{A^{\theta/2}u}^{2}\ds
+\beta(t+\sigma)^{2}\Bigr]\,Q(t)
\;\le\;4\,M(t)\,Q(t),
\end{align*}
the last step because $(T_{1}-t)\n{A^{\theta/2}u_{0}}^{2}\ge0$ on
$[0,T_{2})$. Combining with \eqref{eq:MQ}, for $\alpha:=\frac{p+3}{4}$, so
that $a:=\alpha-1=\frac{p-1}{4}>0$,
$$
M(t)M''(t)-\alpha\bigl(M'(t)\bigr)^{2}
\ \ge\ M(t)\bigl[M''(t)-(p+3)Q(t)\bigr]\ \ge\ 0
\qquad\text{for a.e. }t\in(0,T_{2}).
$$
Set $y:=M^{-a}>0$. Then $y\in C^{1}$, $y'$ is locally absolutely
continuous, and
$$
y''=-a\,M^{-a-2}\bigl[M M''-\alpha(M')^{2}\bigr]\le0
\qquad\text{a.e. on }(0,T_{2}),
$$
so $y'$ is nonincreasing and $y$ is concave on $[0,T_{2})$. Moreover
$M'(0)=2\bigl[(u_{0},u_{1})+\beta\sigma\bigr]>0$ by \eqref{eq:sigma}, hence
$y'(0)=-a\,M(0)^{-a-1}M'(0)<0$, and concavity gives
\begin{equation}
0<y(t)\le\ell(t):=y(0)+y'(0)\,t,\qquad t\in[0,T_{2}).
\label{eq:concave}
\end{equation}
The affine function $\ell$ is strictly decreasing and vanishes at
$$
T^{*}:=\frac{y(0)}{-y'(0)}=\frac{M(0)}{a\,M'(0)}
=\frac{\n{u_{0}}^{2}+T_{1}\n{A^{\theta/2}u_{0}}^{2}+\beta\sigma^{2}}
{\frac{p-1}{2}\bigl[(u_{0},u_{1})+\beta\sigma\bigr]} .
$$
By the definition \eqref{eq:T1} of $T_{1}$ we have $T^{*}=T_{1}$: clearing
denominators, $T^{*}=T_{1}$ is equivalent to
$$
T_{1}\Bigl(\tfrac{p-1}{2}\bigl[(u_{0},u_{1})+\beta\sigma\bigr]
-\n{A^{\theta/2}u_{0}}^{2}\Bigr)=\n{u_{0}}^{2}+\beta\sigma^{2},
$$
which is \eqref{eq:T1} written as an equality.

\emph{Step 5: conclusion.} From \eqref{eq:concave} and $M=y^{-1/a}$,
\begin{equation}
M(t)\ \ge\ \ell(t)^{-1/a},\qquad t\in[0,T_{2}),
\label{eq:Mdiverges}
\end{equation}
and $\ell(t)\downarrow0$ as $t\uparrow T_{1}$.

Suppose, for contradiction, that $T_{\max}>T_{1}$. Then $T_{2}=T_{1}$, so
\eqref{eq:Mdiverges} holds on all of $[0,T_{1})$ and yields
$\lim_{t\uparrow T_{1}}M(t)=+\infty$. On the other hand $[0,T_{1}]$ is a
compact subset of $[0,T_{\max})$, so $R:=\sup_{[0,T_{1}]}\n{\nabla u(t)}$ is
finite by (S1), and every term of \eqref{eq:M} is then bounded on
$[0,T_{1})$: by \eqref{eq:poincare},
$$
M(t)\le\lambda_{1}^{-1}R^{2}
+T_{1}\lambda_{1}^{-(1-\theta)}R^{2}
+T_{1}\n{A^{\theta/2}u_{0}}^{2}
+\beta(T_{1}+\sigma)^{2}<\infty ,
$$
a contradiction. Therefore $T_{\max}\le T_{1}$, which is
\eqref{eq:upperbound}; in particular $T_{\max}<\infty$, so the solution
blows up in finite time in the sense of Definition~\ref{def:solution}.
\end{proof}

The uniform statement follows at once.

\begin{proof}[Proof of Corollary~\ref{cor:uniform}]
By \eqref{eq:interp}, $\n{A^{\theta/2}u_{0}}^{2}\le N_{0}$ for every
$\theta\in[0,1]$, whence $\sigma_{\theta}\le\sigma_{*}$. For
$\sigma>\sigma_{*}$, Theorem~\ref{thm:upper} applies, and since
$x\mapsto\frac{\n{u_{0}}^{2}+\beta\sigma^{2}}{c-x}$ is increasing on
$(-\infty,c)$, replacing $\n{A^{\theta/2}u_{0}}^{2}$ by the larger quantity
$N_{0}$ in \eqref{eq:upperbound} only increases the right-hand side.
\end{proof}

\begin{remark}
The proof uses (S3) only through the inequality it states, and never the
energy identity. Theorem~\ref{thm:upper} therefore applies verbatim to
limits of Galerkin approximations, for which only the energy inequality is
available.
\end{remark}

\section{Lower bound for the blow-up time}\label{sec:lower}

Here the logic is reversed. Rather than being neutralised, the dissipation is
now spent: the nonlinear term is estimated by H\"older's inequality in the
duality $L^{q_{\theta}'}$--$L^{q_{\theta}}$, and Young's inequality then
lets the dissipation absorb the resulting factor $\n{A^{\theta/2}u_{t}}$
entirely. What survives is a closed differential inequality for $\Phi$,
whose solution cannot become infinite before an explicitly computable time.
This is where the strength of the damping is rewarded, through the size of
$q_{\theta}$ and hence through the admissible range \eqref{eq:plower}.

\begin{proof}[Proof of Theorem~\ref{thm:lower}]
Let $n\ge3$; the case $n\le2$ is identical upon replacing
$(q_{\theta},S_{\theta})$ by $(2,\lambda_{1}^{-\theta/2})$, which is
legitimate by the second inequality in \eqref{eq:poincare}. Write
$q:=q_{\theta}$ and $q':=q_{\theta}'$.

Fix $t\in[0,T_{\max})$. By Lemma~\ref{lem:Phiineq} the inequality
\eqref{eq:S4} holds, and we estimate its right-hand side pointwise in $s$.
By H\"older's inequality with the conjugate pair $(q,q')$ and the embedding
\eqref{eq:embed},
$$
2\int_{\Omega}|u|^{p-1}u\,u_{s}\dx
\le2\,\n{u}_{pq'}^{p}\,\n{u_{s}}_{q}
\le2S_{\theta}\,\n{u}_{pq'}^{p}\,\n{A^{\theta/2}u_{s}} ,
$$
and Young's inequality $2ab\le2a^{2}+\frac12b^{2}$ with
$a=\n{A^{\theta/2}u_{s}}$ and $b=S_{\theta}\n{u}_{pq'}^{p}$ gives
$$
2\int_{\Omega}|u|^{p-1}u\,u_{s}\dx
\le2\n{A^{\theta/2}u_{s}}^{2}+\tfrac12S_{\theta}^{2}\n{u}_{pq'}^{2p}.
$$
The first term on the right is absorbed by the dissipation on the left of
\eqref{eq:S4}. Since \eqref{eq:plower} guarantees $pq'\le\frac{2n}{n-2}$ by
\eqref{eq:exprange}, the embedding
$H_{0}^{1}(\Omega)\hookrightarrow L^{pq'}(\Omega)$ gives
$\n{u}_{pq'}^{2p}\le B_{p,\theta}^{2p}\n{\nabla u}^{2p}\le
B_{p,\theta}^{2p}\Phi^{p}$, and therefore
\begin{equation}
\Phi(t)\ \le\ \Phi(0)+K_{p,\theta}\int_{0}^{t}\Phi(s)^{p}\ds,
\qquad t\in[0,T_{\max}),
\label{eq:integralineq}
\end{equation}
with $K_{p,\theta}=\frac12S_{\theta}^{2}B_{p,\theta}^{2p}$ as in
Theorem~\ref{thm:lower}.

We now compare \eqref{eq:integralineq} with an ordinary differential
inequality. Put
$$
G(t):=\Phi(0)+K_{p,\theta}\int_{0}^{t}\Phi(s)^{p}\ds ,
\qquad t\in[0,T_{\max}).
$$
Then $G$ is locally absolutely continuous, nondecreasing, and
$G\ge\Phi(0)>0$; in particular no positivity discussion for $\Phi$ is
needed. By \eqref{eq:integralineq}, $\Phi\le G$, hence
$$
G'(t)=K_{p,\theta}\Phi(t)^{p}\le K_{p,\theta}G(t)^{p}
\qquad\text{for a.e. }t .
$$
Since $G>0$, the function $G^{1-p}$ is locally absolutely continuous with
$$
\bigl(G^{1-p}\bigr)'=(1-p)G^{-p}G'\ge-(p-1)K_{p,\theta}
\qquad\text{a.e.},
$$
and integrating from $0$ to $t$, with $G(0)=\Phi(0)$,
\begin{equation}
G(t)^{1-p}\ \ge\ \Phi(0)^{1-p}-(p-1)K_{p,\theta}\,t,
\qquad t\in[0,T_{\max}).
\label{eq:Gbound}
\end{equation}
Let
$$
t_{c}:=\frac{1}{(p-1)K_{p,\theta}\Phi(0)^{p-1}},
$$
so that the right-hand side of \eqref{eq:Gbound} is strictly positive
exactly for $t<t_{c}$. For $t\in[0,T_{\max})\cap[0,t_{c})$ we may therefore
raise \eqref{eq:Gbound} to the power $-\frac{1}{p-1}$ and use $\Phi\le G$:
\begin{equation}
\Phi(t)\ \le\ G(t)\ \le\
\bigl[\Phi(0)^{1-p}-(p-1)K_{p,\theta}t\bigr]^{-\frac{1}{p-1}} .
\label{eq:Phibound}
\end{equation}
The right-hand side of \eqref{eq:Phibound} is finite and nondecreasing on
$[0,t_{c})$, so $\Phi$ is bounded on every compact subinterval of
$[0,T_{\max})\cap[0,t_{c})$. If we had $T_{\max}<t_{c}$, then $\Phi$ would
be bounded on all of $[0,T_{\max})$, contradicting the blow-up
alternative (S4). Hence $T_{\max}\ge t_{c}$, which is
\eqref{eq:lowerbound}.
\end{proof}

\begin{remark}
The proof uses no information on the initial energy, and in particular
applies both to solutions that blow up and to global ones, for which
\eqref{eq:lowerbound} is vacuous. Like the proof of
Theorem~\ref{thm:upper}, it uses (S3) only as an inequality, here through
Lemma~\ref{lem:Phiineq}. The estimate \eqref{eq:Phibound} also gives an
explicit growth envelope for the mechanical energy on the guaranteed
interval of existence, not merely the endpoint $t_{c}$.
\end{remark}

\section{Conclusion}\label{sec:conclusion}

For the semilinear wave equation \eqref{eq:main} with fractional structural
damping we have established two explicit lifespan estimates valid on the
whole scale $\theta\in[0,1]$, and hence for the entire family of dissipation
mechanisms interpolating between external friction and internal
Kelvin--Voigt viscoelasticity. Theorem~\ref{thm:upper} gives finite-time
blow-up together with the upper bound \eqref{eq:upperbound} for data with
energy below the depth of the potential well and negative Nehari functional,
and Corollary~\ref{cor:uniform} converts it into a bound which no longer
involves $\theta$ at all. Theorem~\ref{thm:lower} gives the lower bound
\eqref{eq:lowerbound} for exponents in the range \eqref{eq:plower}, with all
constants identified as norms of explicit embeddings. Corollary
\ref{cor:twosided} combines them into a bracket for the lifespan whose two
ends are computable from the initial data.

The two bounds treat the dissipation in opposite ways, and this is the
structural point of the paper. In Theorem~\ref{thm:upper} the damping does
not enter the estimate: the functional \eqref{eq:M} is designed so that it
cancels identically in Step 2 of the proof, which is why the upper bound has
the same form for every $\theta$ and why the uniform version is available.
In Theorem~\ref{thm:lower} the damping is used quantitatively, and the
stronger high-frequency dissipation associated with larger $\theta$ is
rewarded by the wider admissible range \eqref{eq:plower} of nonlinear loads.
In engineering terms, the first bound predicts failure and the second
certifies a safe interval, and only the second improves as the internal
damping of the material is stiffened.

\begin{remark}\label{rem:abstract}
The proofs use no structure of $-\Delta$ beyond self-adjointness, positivity
and compactness of the resolvent. Both theorems therefore extend verbatim to
$$
u_{tt}+A^{\theta}u_{t}+Au=|u|^{p-1}u
$$
for any operator $A$ of that type on $L^{2}(\Omega)$ whose form domain
embeds in $L^{p+1}(\Omega)$, with $q_{\theta}$ dictated by the embedding of
$D(A^{\theta/2})$ and with Definition~\ref{def:solution} adapted
accordingly. This covers the operators that occur in structural models: the
linear elasticity operator with structural damping, and the fourth-order
plate and beam operators $A=\Delta^{2}$ with internal damping
$\Delta^{2\theta}u_{t}$, for which $\theta=\frac12$ corresponds to the
square-root damping of Chen and Triggiani
\cite{ChenTriggiani1989,ChenTriggiani1990} and $\theta=1$ to Kelvin--Voigt
damping of the plate. In each case the two bounds retain the form
\eqref{eq:upperbound} and \eqref{eq:lowerbound}; only the three embedding
constants \eqref{eq:Sstar} and \eqref{eq:constants} and the first eigenvalue
must be recomputed for the operator and geometry at hand.
\end{remark}

We close with four questions suggested by the present analysis. First,
whether the lower bound \eqref{eq:lowerbound} survives in the range
$\frac{n+2\theta}{n-2}<p<\frac{n+2}{n-2}$ for $\theta\in[0,1)$, that is,
whether the restriction \eqref{eq:plower} is an artefact of the method.
Second, whether the constant $K_{p,\theta}$ is monotone in $\theta$, which
is what would be needed to turn Remark~\ref{rem:comments}(iii) into a
monotonicity statement for the lifespan bound itself. Third, the local
theory in the energy space for $\theta\in[0,\frac12)$ and
$\frac{n}{n-2}<p<\frac{n+2}{n-2}$, discussed in
Remark~\ref{rem:lwpdiscussion}. Fourth, a numerical study of the sharpness
of the bracket of Corollary~\ref{cor:twosided} as a function of $\theta$,
which would require a spectral Galerkin discretisation of $A^{\theta}$ and
would complement the computation carried out at $\theta=1$ in
\cite{BchatniaHamoudaKaabi2026}. Let us add that a complement to
Theorem~\ref{thm:upper} in the stable regime, $E(0)<d$ and $I(u_{0})>0$, is
expected to give global existence, as in
\cite{PayneSattinger1975,GazzolaSquassina2006}; since the focus here is on
the blow-up time we do not pursue it.

\section*{Declarations}

\noindent\textbf{Conflict of interest.} The author declares that he has no
conflict of interest.

\smallskip
\noindent\textbf{Data availability.} No datasets were generated or analysed
during the current study.

\end{document}